\documentclass{article}

     \usepackage[preprint,nonatbib]{neurips_2019}

\usepackage[utf8]{inputenc} % allow utf-8 input
\usepackage[T1]{fontenc}    % use 8-bit T1 fonts
\usepackage{hyperref}       % hyperlinks
\usepackage{url}            % simple URL typesetting
\usepackage{booktabs}       % professional-quality tables
\usepackage{amsfonts}       % blackboard math symbols
\usepackage{nicefrac}       % compact symbols for 1/2, etc.
\usepackage{microtype}      % microtypography
\usepackage{algorithm}
\usepackage{algorithmicx}
\usepackage{algpseudocode}

\usepackage{amssymb,amsmath,amsthm}
\usepackage[centercolon]{mathtools}
\usepackage{graphics}
\RequirePackage[labelformat=simple]{subcaption}

\newcommand{\Llog}{L\log L}
\newcommand{\Lexp}{L_{\exp}}
\newcommand{\bd}{\,{\mathrm{d}}}

\newcommand{\RR}{{\mathbb{R}}}

\newcommand{\meas}{{\mathcal{M}}}

\newcommand{\prob}{{\mathcal{P}}}
\newcommand{\cont}{{\mathcal{C}}}

\newcommand{\ppfw}[1][i]{(P_{#1})_\#}
\newcommand{\weakstarto}{\xrightharpoonup*}

\newcommand{\norm}[1]{\|{#1}\|}
\newcommand{\scp}[2]{\langle{#1},{#2}\rangle}
\newcommand{\abs}[1]{\left|{#1}\right|}
\newcommand{\ones}{\mathbf{1}}

\DeclareMathOperator{\diag}{diag}
\DeclareMathOperator{\supp}{spt}
\newcommand\spt\supp

\usepackage[nameinlink,capitalize]{cleveref}
\newtheorem{theorem}{Theorem}[section]
\newtheorem{proposition}[theorem]{Proposition}
\newtheorem{definition}[theorem]{Definition}
\newtheorem{lemma}[theorem]{Lemma}
\newtheorem{remark}[theorem]{Remark}
\newtheorem{example}[theorem]{Example}

\newtheorem{assumption}[theorem]{Assumption}
\numberwithin{equation}{section}

\usepackage{backwards}

\usepackage[svgnames]{xcolor}
\usepackage[textsize=footnotesize,colorinlistoftodos]{todonotes}

\title{Orlicz-space regularization for optimal transport and algorithms for quadratic regularization}

\author{%
  Dirk Lorenz\\
  Institute of Analysis and Algebra\\
  TU Braunschweig\\
  38092 Braunschweig, Germany\\
  \texttt{d.lorenz@tu-braunschweig.de}\\
  \And
  Hinrich Mahler\\
  Institute of Analysis and Algebra\\
  TU Braunschweig\\
  38092 Braunschweig, Germany\\
  \texttt{h.mahler@tu-braunschweig.de}\\
}

\begin{document}

\maketitle

\begin{abstract}
  We investigate the continuous optimal transport problem in the so-called Kantorovich form, i.e. given two Radon measures on two compact sets, we seek an optimal transport plan which is another Radon measure on the product of the sets that has these two measures as marginals and minimizes a certain cost function.
  We consider regularization of the problem with so-called Young's functions, which forces the optimal transport plan to be a function in the corresponding Orlicz space rather than a Radon measure. We derive the predual problem and show strong duality and existence of primal solutions to the regularized problem. Existence of (pre-)dual solutions will be shown for the special case of $L^p$ regularization for $p\geq2$. Then we derive four algorithms to solve the dual problem of the quadratically regularized problem: A cyclic projection method, a dual gradient decent, a simple fixed point method, and Nesterov's accelerated gradient, all of which have a very low cost per iteration.
\end{abstract}

\section{Introduction}

We consider the optimal transport problem in the following form: For compact sets $Ω_1,\,Ω_2\subset ℝ^n$, measures $\mu_{1},\mu_{2}$ on $\Omega_{1},\Omega_{2}$, respectively, with the same total mass and a real-valued cost function $c:\Omega_{1}\times\Omega_{2}\to ℝ$ we want to solve
\[
	\inf_{\pi}\int_{\Omega_{1}\times\Omega_{2}}c\dpi
\]
where the infimum is taken over all measures on $\Omega_{1}\times \Omega_{2}$ which have $\mu_{1}$ and $\mu_{2}$ as their first and second marginals, respectively (see~\cite{peyre2019computational,villani2003topics}). Since optimal plans $\pi$ tend to be singular measures (even for marginals with smooth densities \cite{villani2008optimal,Santambrogio}), regularization of the problem have become more important, most prominently entropic regularization~\cite{carlier2017convergence,cuturi2013sinkhorn,Benamou:2015,cuturi2016smoothed,clason:2019} which ensures that optimal plans have densities. It has been shown in~\cite{clason:2019} that the analysis of entropically regularized optimal transport problems naturally takes place in the function space $L\log L$ (also called Zygmund space\cite{Bennett:1988}) and that optimal plans for entropic regularization are always in $L\log L(\Omega_{1}\times \Omega_{2})$ and exist if and only if the marginals are in the spaces $L\log L(\Omega_{i})$. These spaces are an example of so-called Orlicz spaces~\cite{Rao:1988} and hence, we consider regularization in these spaces in this paper. Another motivation to study a more general regularization comes from the fact that regularization with the $L^{2}$-norm have been shown to be beneficial in some applications, see~\cite{roberts2017gini,blondel2017smooth,dessein2018regularized,Lorenz:2019}. 

To simplify notation, denote $Ω:= Ω_1 \times Ω_2$. The regularized problems we consider are
\begin{equation}\tag{P}\label{eq:reg_kantorovich}
	\inf_{\substack{\pi\geq 0,\\\ppfw[i]{π} = {μ_i},\,i=1,2}} \int_{\Omega} c \dpi + γ \int_{\Omega} Φ\compos π \commath
\end{equation}%
where $γ>0$ and the infimum is taken over all positive measure, which have densities with respect to the Lebesgue measure and the constraints $\ppfw[i]{π} = {μ_i}$ state that $\pi$ should have the marginals $\mu_{i}$. The two main cases we will consider in this paper are $\Phi(t) = t\log t$ (entropic regualarization) and $\Phi(t) = \nicefrac{t^2}2$ (quadratic regularization), however, many results in this paper hold for  more general functions $\Phi$.

The rest of the paper is organized as follows. First, we introduce the so-called Young's functions and Orlicz spaces in \cref{sec:yf_os}. Moreover, a slight generalization of Young's functions is defined. In \cref{sec:existence} we give a strong duality result for \eqref{eq:reg_kantorovich}, which allows us to prove existence of primal solutions. Existence of solutions for the (pre-)dual problem will be discussed for the special case $Φ(t) = \nicefrac{(t_+)^p}{p}$ for $p\geq 2$. Due to limited space, we omit some proofs in this section and the proofs will be published in an extended version. Finally, we consider numerical methods for the quadratic regularization in \cref{sec:numerics}. Here, we will work with the duality result presented in \cref{sec:existence} and concentrate on algorithms with very low cost per iteration.

\section{Young's Functions \& Orlicz Spaces}\label{sec:yf_os}

%\subsection{Young's Functions}\label{sec:yf_os}
We briefly introduce some notions about Young's functions and Orlicz spaces. For a more detailed introduction, see~\cite{Bennett:1988,Rao:1988}.
\begin{definition}[Young's function {\cite[Definition IV.8.1]{Bennett:1988}}, {\cite[Eq. (2.22)]{cavaliere:2015}}\,]\label{def:yf}
	\begin{enumerate}
		\item Let $ϕ : [0, ∞) →{} [0, ∞]$ be increasing and lower semicontinuous,	with $ϕ(0) = 0$. Suppose that $ϕ$ is neither identically zero nor	identically infinite on $(0,∞)$. Then the function $Φ$, defined by
		\[
		Φ(t) := \int_0^t ϕ(s) \ds\commath
		\]
		is said to be a \emph{Young's function}.
		
		\item A Young's function $Φ$ is said to have the $\Delta_2$-property near infinity if $Φ(t)<∞$ for all $t$ and
		\[
		∃C>0,\,t_0\geq 0:\,∀t\geq t_0:\,Φ(2t) \leq C Φ(t).
		\]
	\end{enumerate}
\end{definition}

By definition, Young's functions are convex and for a Young's function $\Phi$ it holds that the \emph{complementary Young's function} $\Psi(x) := \int_{0}^{x}(\Phi')^{-1}(s)\ds$ is also a Young's function. Indeed, the complementary Young's function is related to the convex conjugate $Φ^*$.

The negative entropy regularization uses the regularization functional $\int_{\Omega}\pi\log\pi$ and the function $t\mapsto t\log t$ is not a Young's function. Hence, we introduce a slight generalization.

\begin{definition}[Quasi-Young's functions]\label{def:qyfs}
	Let $Φ$ be a Young's function and $t_{0}\geq 0$. Let $\check{Φ}$ be a convex, lower semicontinuous function bounded from below with
	\[
		\check{Φ}(t) = Φ(t)-Φ(t_{0})\qquad\forall t\geq t_{0}
	\]
	and $\check{Φ}(t)\leq 0$ for all $t<t_{0}$. Then $\check{Φ}$ is said to be a \emph{quasi-Young's function} induced by $Φ$.
\end{definition}

\begin{example}
	The function $\check{Φ}(t) = t\log_+ t$ is a quasi-Young's function induced by the Young's function $Φ(t) = t\log t$, with $t_{0} = 1$. It holds $Φ(t) = \max\{0,\check{Φ}(t)\}$.
\end{example}

\begin{definition}[Orlicz spaces {\cite[Definition IV.8.10]{Bennett:1988}}\,]
	Let $Φ$ be a Young's function and $Ω⊂ℝ^n$. Define the \emph{Luxemburg norm} of a measurable function $f: Ω → ℝ$ as
	\[
		\lux{f} := \inf \set{γ\geq0 }{ \int_{\Omega} Φ\compos\of{\frac{\abs{f}}{γ}}\dleb \leq 1}\commath[.]
	\]
	Then the space
	\[
		\Lorl(Ω) := \set{f:Ω→ℝ\,\text{measurable}}{\lux f < ∞}
	\]
	of measurable functions on $Ω$ with finite Luxemburg norm is called the \emph{Orlicz space of $Φ$}.
\end{definition}

One can verify that the definitions of $\lux{\blank}$ and $\Lorl$ are
essentially independent of whether $Φ$ is a Young's function or just a quasi-Young's function. To
simplify notation, $\lux[\tilde{Φ}]{\blank}$ and $\Lorl[\tilde{Φ}]$
will therefore also be used for quasi-Young's function{s} $\tilde{Φ}$. Note that for a
quasi-Young's function $\tilde{Φ}$ induced by a Young's function $Φ$, it holds that $\Lorl[\tilde{Φ}] = \Lorl$,
while in general $\lux[\tilde{Φ}]{\blank}$ and $\lux{\blank}$ are
equivalent but not equal.

It is well known that for a quasi-Young's function $\Phi$ with $\lim_{t→∞}\nicefrac{Φ(t)}{t} = ∞$ and bounded domain $\Omega$ it holds that  $\Lorl(Ω) ⊂ \L 1(Ω)$.
Moreover for complementary Young's function{s} $Φ$ and $Ψ$ which are  proper, locally integrable and have the $\Delta_2$-property near infinity it holds that $(\Lorl{})^*$ is canonically isometrically isomorphic to $(\Lorl[Ψ],\lux[Ψ]{\blank})$ (see, e.g.~\cite{diening:2011}).
	
\begin{example}[$\LlogL{}$ and $\Lexp{}$]
	Let $Φ(t) = t \log t$ and $\tilde{Φ}(t) = t \log_+ t$. The space of measurable functions $f$ with $\int_{\Omega} \tilde{Φ}\compos \abs{f}\dleb<∞$ is called $\LlogL{}$. Because $\Lorl = \Lorl[\tilde{Φ}]$, the space of measurable functions $g$ with $\int_{\Omega} {Φ}\compos\abs{g}\dleb<∞$ is equal to  $\LlogL$ as well. The complementary Young's function $\tilde{Ψ}$ of $\tilde{Φ}$ is given by
	\[
		\tilde{Ψ}(t) =
			\begin{cases}
				t,&t\leq 1\\
				\e^{t-1},&\text{else.}
			\end{cases}
	\]
	Hence, the dual space of $\LlogL{}$ is given by the space of measurable functions $h$ that satisfy $\int_{\Omega} \tilde{Ψ}\compos\abs{h}\dleb<∞$, which is called $\Lexp{}$.
\end{example}

Similar to \cite[Lemma 2.11]{clason:2019} one can show that the marginals of a transport plan $π\in\Lorl$ are also in $\Lorl$:
\begin{lemma}\label{thm:proj_contraction}
	If $π\in\Lorl(Ω)$ for a quasi-Young's function $Φ$, then $\ppfw π\in\Lorl(Ω_i)$ for $i=1,2$ with
	\[
		\lux{\ppfw π} \leq \max\of{1, \abs{Ω_{3-i}}} \lux{π}\commath[.]
	\]
\end{lemma}

Conversely, for the direct product of two marginals to lie in $\Lorl(Ω)$, some assumptions have to be made.
\begin{proposition}\label{thm:tensor_feasible}
	Let $Ω_1, Ω_2⊂ℝ^n $ be bounded and $Φ$ be a quasi-Young's function satisfying either
	\begin{equation}
		\label{cond:tensor_feasible:1}
		Φ(xy) \leq C Φ(x)Φ(y)
	\end{equation}
	for some $C>0$ or
	\begin{equation}\label{cond:tensor_feasible:2}
		\begin{split}
			Φ(xy) &\leq C_1 xΦ(y) + C_2 Φ(x)y\\
			\frac{Φ(t)}{t} & \xrightarrow[t\to\infty]{} ∞
		\end{split}
	\end{equation}
	for some $C_1,C_2 \geq 0$.
	If $π = μ_1\dprod μ_2$ and $μ_i\in\Lorl(Ω_i)$ for $i=1,2$, where $(μ_1\dprod μ_2)(x_1,x_2):=μ_1(x_1)μ_2(x_2)$, then $π\in\Lorl(Ω)$.
\end{proposition}

\begin{example}
	\begin{enumerate}
		\item For $Φ(t) = \frac{t^p}{p}$, $p>1$, \cref{cond:tensor_feasible:2,cond:tensor_feasible:1} both hold trivially.
		\item For $Φ(t) = t\log t$, \cref{cond:tensor_feasible:2} holds, since $\log(xy) = \log(x) + \log(y)$.
	\end{enumerate}
\end{example}

\section{Existence of Solutions}\label{sec:existence}

In this section, strong duality will be shown for the regularized mass transport \eqref{eq:reg_kantorovich} using Fenchel duality in the spaces $\radon(Ω)$ and $\CC(Ω)$. Here, the general framework as outlined in \eg \cite[Chap. 9]{Attouch:2006} or \cite[Sec. III.4]{Ekeland:1999}, is used. The result will then be used to study the question of existence of solutions for both the primal and the dual problem.

\begin{theorem}[Strong duality]\label{thm:str:dual}
	Let $Φ$ be a quasi-Young's function and let
	\[
		\tilde{Φ}(t) :=
			\begin{cases}
				∞,&t<0,\\
				Φ(t),&\text{else}.
			\end{cases}
	\]
	If $\tilde{Φ}^*\compos\of{\nicefrac{-c}{γ}}$ is integrable, then the predual problem to \eqref{eq:reg_kantorovich} is
	\begin{equation}\tag{P*}\label{eq:dual_reg_kantorovich}
		\sup_{\substack{α_i\in\CC(Ω_i),\\i=1,2}} \int_{\Omega_1} α_1 \dμ_1 + \int_{\Omega_2}α_2\dμ_2
		- γ\int_{\Omega} \tilde{Φ}^*\compos\of{\frac{{α_1}\dplus{α_2}-c}{γ}} \dleb\,,
	\end{equation}
	where $(α_1\oplus α_2)(x_1,x_2):=α_1(x_1) + α_2(x_2)$,
	and strong duality holds. Furthermore, if the supremum is finite, \eqref{eq:reg_kantorovich} posseses a minimizer.
\end{theorem}
\begin{proof}
	This proof follows the outline of the proof in {\cite[Theorem 3.1]{clason:2019}}.	
	First note that	$\radon(\Omega)$ is the dual space of $\CC(\Omega)$ for compact	$\Omega$.
	Furthermore Slater's condition is fulfilled with ${α_1},{α_2} = 0$ so that strong duality holds and (assuming a finiteness of the supremum) the primal \eqref{eq:reg_kantorovich} possesses a minimizer. Additionally, the integrand of the last integral in \eqref{eq:dual_reg_kantorovich} is normal, so that it can be conjugated pointwise \cite[Theorem 2]{rockafellar:1968}.

	Carrying out the conjugation one obtains
	\begingroup%
	\allowdisplaybreaks%
	\begin{align*}
		\sup_{\substack{α_i\in\CC(Ω_i),\\i=1,2}} &\int_{\Omega_1} {α_1} \dμ_1 + \int_{\Omega_2}{α_2}\dμ_2
		- γ\int_{\Omega} \tilde{Φ}^*\compos\of{\frac{α_1\dplus α_2 - c}{γ}} \dleb\\
		%
		%			&\begin{multlined}
		%				=\sup_{\substack{α_i\in\CC(Ω_i),\\i=1,2}} \int_{\Omega_1} {α_1} \dμ_1 + \int_{\Omega_2}{α_2}\dμ_2
		%				 - γ\int_{\Omega} \tilde{Φ}^*\compos\of{\frac{α_1\dplus α_2 - c}{γ}} \dleb
		%			\end{multlined}\\
		%			
		&\begin{multlined}
			=\sup_{\substack{α_i\in\CC(Ω_i),\\i=1,2}} \int_{\Omega_1} {α_1} \dμ_1 + \int_{\Omega_2}{α_2}\dμ_2 + \int_{\Omega} \of{\inf_{π} \of{c - α_1\dplus α_2} π + γ\tilde{Φ}\compos π }\dleb
		\end{multlined}\\
		&\begin{multlined}
			=\sup_{\substack{α_i\in\CC(Ω_i),\\i=1,2}} \inf_{π} \int_{\Omega} cπ + γ \tilde{Φ}\compos π \dleb
			+ \int_{\Omega_1} {α_1}\d ({μ_1}-\ppfw[1]π)\\
			+ \int_{\Omega_2} {α_2}\d({μ_2}-\ppfw[2]π)
		\end{multlined}\\
		&=\inf_{\substack{π,\\\ppfw[i]{π} = {μ_i},\,i=1,2}} \int_{\Omega} c \dpi + γ \int_{\Omega} \tilde{Φ}\compos π \dleb\\
		&=\inf_{\substack{0\leqπ,\\\ppfw[i]{π} = {μ_i},\,i=1,2}} \int_{\Omega} c \dpi + γ \int_{\Omega} Φ(π) \dleb\commath
	\end{align*}
	\endgroup%
	which is \eqref{eq:reg_kantorovich}.
\end{proof}

\begin{example}
	\begin{enumerate}
		\item Using $Φ(t) = t\log t$, one obtaines the result for $\LlogL{}$ as stated in \cite[Theorem 3.1]{clason:2019}
		\item Using $Φ(t) = \nicefrac{t^2}2$, one obtains the result for $\L2$ as stated in \cite{Lorenz:2019}.
	\end{enumerate}
\end{example}

\begin{remark}
	\cref{thm:str:dual} does \emph{not} claim that the supremum is attained, i.e. the predual problem \eqref{eq:dual_reg_kantorovich} admits a solution. Moreover, the solutions of \eqref{eq:dual_reg_kantorovich} cannot be unique since one can add and subtract constants to $α_1$ and $α_2$,	respectively, without changing the functional value.
\end{remark}

\subsection{Existence Result for the Primal Problem}

The duality result can now be used to address the question of existence of a solution to \eqref{eq:reg_kantorovich}.

\begin{theorem}%[Existence of solutions of \eqref{eq:reg_kantorovich}]
	\label{thm:dual:prop_solution}
	Problem \eqref{eq:reg_kantorovich} admits a minimizer $\bar{π}$ if and only if $μ_i \in\Lorl(Ω_i)$ for $i=1,2$ and
	\begin{equation}\label{cond:prod_of_fcts}
		μ_1\dprod μ_2 \in\Lorl(Ω)\quad \forall μ_i \in\Lorl(Ω_i),\,i=1,2 \commath[.]
	\end{equation}
	In this case, $\bar{π} \in\Lorl(Ω)$. Moreover, the minimizer is unique, if $Φ$ is strictly convex.
\end{theorem}
\begin{proof}
	The proof given in \cite[Theorem 3.3]{clason:2019} for $Φ(t) = t \log t$ holds for arbitrary $Φ$.
	That is, the necessity of the condition $μ_i \in\Lorl(Ω_i)$, $i=1,2$ only relies on \cref{thm:proj_contraction}. For sufficiency, it is noted that for $μ_i \in\Lorl(Ω_i)$, $i=1,2$ it holds that $π=μ_1\dtimes μ_2\in\Lorl(Ω)$, which is ensured by \cref{cond:prod_of_fcts}. Thus, the infimum in \eqref{eq:reg_kantorovich} is finite and weak duality shows that the supremum in \eqref{eq:dual_reg_kantorovich} is finite as well. Existence of a solution for \eqref{eq:reg_kantorovich} now follows from \cref{thm:str:dual}.
	
	If strict convexity holds for $Φ$, it directly implies uniqueness.
\end{proof}

\begin{remark}
	For example, \cref{cond:prod_of_fcts} is satisfied when $Φ$ satisfies either \cref{cond:tensor_feasible:1} or \cref{cond:tensor_feasible:2} since in those cases \cref{thm:tensor_feasible} holds.
\end{remark}

\subsection{Existence Result for the Predual Problem}

The question of existence of solutions to the predual problem \eqref{eq:dual_reg_kantorovich} proves to be more difficult for general Young's functions. 
There are results that shows existence for the predual problem in the entropic case~\cite{clason:2019} and in the quadratic case~\cite{Lorenz:2019}, but their proofs are quite different in nature.
Here, we only treat Young's functions of the type $Φ(t)= \nicefrac{t^p}{p}$ for $p>1$, i.e, only regularization in $L^{p}$. Note that in this case ${Φ}^*(t) = \nicefrac{(\pos{t})^q}{q}$, where $\nicefrac 1 p + \nicefrac 1 q = 1$ and the predual is actually the dual.

\begin{assumption}\label{asspt:dual}
	Let $Ω_1$ and $Ω_2$ to be compact comains, let the cost function $c$ be continuous and fulfill $c \geq c^{\dagger} > -∞$. Furthermore, the marginals ${μ_i}\in\Lp(Ω_i)$ satisfy
	%${μ_i}\in\CC(Ω_i)$ and
	${μ_i}\geq δ > 0$
	a.e.
	for $i=1,2$ and finally assume that $\int_{\Omega_1} {μ_1} =\int_{\Omega_2} {μ_2} = 1$.
\end{assumption}

It can not be expected for \eqref{eq:dual_reg_kantorovich} to have continuous solutions ${α_1}$, ${α_2}$. However, observe that the objective function of \eqref{eq:dual_reg_kantorovich} is also well defined for functions ${α_i}\in\L{1}(Ω_i)$, $i=1,2$, with $\pos{({α_1}\oplus{α_2} - c)}/γ \in \L{q}(Ω)$. This gives rise to the following variant of the dual problem, for which existence of minimizers can be shown:

\begin{equation}\tag{P$^\dagger$}\label{eq:dual_aux}
	\begin{multlined}
		\min\left\{%
			Λ({α_1},{α_2}) := \frac{1}{q} \norm{\pos{({α_1}\oplus {α_2} -c)}}_q^q
			\vspace*{-1em}%
			- γ^{q-1}\int_{\Omega_1} {α_1} \dμ_1 - γ^{q-1}\int_{\Omega_2}{α_2}\dμ_2
			\right.\\
			\left.\vphantom{\int_{\Omega_2}{α_2}}\middle|\,%
			α_i\in\L{1}(Ω_i),\,i=1,2,\,\pos{({α_1}\oplus{α_2} - c)}/γ \in \L{q}(Ω)%
		\right\}%
	\end{multlined}
\end{equation}
The strategy is now as follows.

\begin{enumerate}
	\item First, show that \eqref{eq:dual_aux} admits a solution $({\bar{α}_1},{\bar{α}_2}) \in \L{1}(Ω_1) \times \L{1}(Ω_2)$.
	\item Then, prove that ${\bar{α}_1}$ and ${\bar{α}_2}$ possess higher regularity, namely that they are functions in $\L{q}(Ω_i)$.
\end{enumerate}

The objective function is extended to allow to deal with weakly-$*$ converging sequences. To that end, define
\[
	G:\L{q}(Ω)\ni w \mapsto \int_{\Omega} \of{\frac{1}{q} (\pos{w})^q - w μ }\dleb \inℝ \commath
\]
where $μ := γ^{q-1}({μ_1}\otimes {μ_2})$. Then, thanks to the normalization of ${μ_1}$ and ${μ_2}$,
\[
	Λ({α_1},{α_2}) = G({α_1}\oplus {α_2} - c) - \int_{\Omega} c μ \dleb\quad ∀{α_1},{α_2} \in\L{q}\commath[.]
\]
Of course, $G$ is also well defined as a functional on the feasible set of \eqref{eq:dual_aux} and this functional will be denoted by the same symbol to ease notation. In order to extend $G$ to the space of Radon measures, consider for a given measure $w \in\radon(Ω)$, the Hahn-Jordan decomposition $w=\pos w + \negprt w$ and assume $\pos w \in\Lq(Ω)$. Then, set
\[
	G(w) := \int_{\Omega} \frac{1}{q} (\pos{w})^q \dleb - \int_{\Omega} μ \dw\commath[.]
\]
With slight abuse of notation, this mapping will be denoted by $G$, too.

\begin{remark}
	If $w\acl$, then $\pos w \in\L1(Ω)$ and $\pos w(x) = \max\sset{0, w(x)}$ $\leb$-a.e. in $Ω$. Hence, both functionals denoted by $G$ conincide on $\Lq(Ω)$, which justifies this notation.
\end{remark}
The following auxiliary results are generalizations of the corresponding results in \cite{Lorenz:2019} and can be proven with little effort.

\begin{lemma}\label{thm:dual:boundedness}
	Let \cref{asspt:dual} hold and suppose that a sequence $\seq w n⊂ \L q (Ω)$ fulfills
	\[
		G(w_n) \leq C < ∞\quad ∀ n \inℕ  
	\]
	for some $C>0$.
	Then, the sequences $\pos{\seq w n}$ and $\negprt{\seq w n}$ are bounded in $\L q(Ω)$ and $\L 1 (Ω)$, respectively.
\end{lemma}
\begin{proof}
	The assertion w.r.t. $\pos{\seq w n}$ can be proven by the same argument used in \cite[Lemma 2.6]{Lorenz:2019}. The second one can be seen by making use of $μ\geq δ$ with $δ$ from \cref{asspt:dual}, which yields the estimate
	
	\begin{align*}
		C \geq G(w_n) & = \frac{1}{q}\int_{\Omega} \pos{(w_n)}^q \dleb - \int_{\Omega} \pos{(w_n)} μ\dleb + \int_{\Omega} \negprt{(w_n)} μ\dleb\\
		&\geq \frac{1}{q} \norm{\pos{(w_n)}}_q^q - \norm{μ}_p\norm{\pos{(w_n)}}_q + γ^{q-1}δ^2 \norm{\negprt{(w_n)}}_1\\
		&\geq  - \norm{μ}_p\norm{\pos{(w_n)}}_q + γ^{q-1}δ^2 \norm{\negprt{(w_n)}}_1\commath[.]
	\end{align*}
	
	Since $\norm{\pos{(w_n)}}_q$ is already known to be bounded, the second assertion holds.
\end{proof}

\begin{lemma}\label{thm:dual:wsconv}
	Let \cref{asspt:dual} hold and a sequence $\seq w n⊂ \L q (Ω)$ be given such that $w_n \weakstarto \bar{w}$ in $\meas(Ω)$ and $G(w_n) \leq C < ∞$ for all $n\inℕ$. Then it holds that $\pos{\bar w}\in\L q (Ω)$ and
	\[
		G(\bar{w}) \leq \liminf_{n→∞} G(w_n).
	\]
\end{lemma}

\begin{proposition}\label{thm:dual-l1-solultion}
	Let \cref{asspt:dual} hold. Then, \eqref{eq:dual_aux} admits a solution $({\bar{α}_1},{\bar{α}_2})\in\L 1 (Ω_1)\times \L 1 (Ω_2)$.
\end{proposition} 
\begin{proof}
	In \cite[Proposition 2.10]{Lorenz:2019} the statement is proven for $p=2$ via the classical direct method of the calculus of variations using only
	\cite[Lemmas 2.8 \& 2.9]{Lorenz:2019} and \cref{thm:dual:boundedness,thm:dual:wsconv}, where
	\cite[Lemmas 2.8 \& 2.9]{Lorenz:2019} are rather technical results holding independently of the choice of $Φ$.
	Hence, the proof also holds for $p\geq 2$.
\end{proof}

The next results states that $α_i$, $i=1,2$ are indeed functions in $\Lq(Ω_i)$.

\begin{theorem}\label{thm:dual:lq}
	Let \cref{asspt:dual} hold and let $p\geq 2$. Then every optimal solution $({\bar{α}_1}, {\bar{α}_2})$ from \cref{thm:dual-l1-solultion} satisfies ${\bar{α}_i}\in\Lq(Ω_i)$, $i=1,2$.
\end{theorem}

\section{Numerical Methods for Quadratic Regularization}\label{sec:numerics}

In this section we turn to numerical methods and focus only on the case of quadratic regularization. For the special case of the negative entropy, \ie $Φ(t) = t\log t$, there is the celebrated Sinkhorn method~\cite{cuturi2013sinkhorn,Sinkhorn:1966,Sinkhorn:1967a} which can be interpreted as an alternating projection method~\cite{Benamou:2015}. For the case of quadratic regularization, i.e. $\Phi(t) = t^{2}/2$~\cite{Lorenz:2019} proposed a Gauß-Seidel method (which is similar to the Sinkhorn method) and a semismooth Newton method (which is similar to the Sinkhorn-Newton method from \cite{brauer2017sinkhorn} for entropic regularization). Both methods converge reasonably well, but the iterations become expensive for large scale problems. In~\cite{blondel2017smooth} used the standard solver L-BFGS method to solve the dual problems which also works good for medium scale problems, but is not straightforward to parallelize. Here we focus on methods that come with very low cost per iteration and which allow for simple parallelization.

We switch to the discrete case and slightly change notation. The marginals are two non-negative vectors $\mu\in\RR^{N}$ and $\nu\in\RR^{M}$ with $\sum_{i}\mu_{i} = \sum_{j}\nu_{j}$ and the cost is $c\in\RR^{N\times M}$. We denote by $\ones$ the vector of all ones (of appropriate size). A feasible transport plan is now a matrix $\pi\in\RR^{N\times M}$ with $\pi\ones = \mu$ (matching row-sums) and $\pi^{T}\ones = \nu$ (matching colum sums). The quadratically regularized optimal transport problem is then, for some $\gamma>0$
\begin{equation}\label{eq:qrot-discrete}
  \min_{\substack {\pi\geq 0\\\pi\ones=\mu\\\pi^{T}\ones = \nu}}\sum_{ij}c_{ij}\pi_{ij} + \tfrac\gamma2\norm{\pi}_{2}^{2}.
\end{equation}
The starting point for our algorithms for the quadratically regularized problem is the optimality system: $\pi$ is optimal if and only if there are two vectors $\alpha\in\RR^{N}$ and $\beta\in\RR^{M}$ such that
\begin{align*}
  \pi & = (\alpha\oplus\beta - c)_{+}/\gamma\\
  \sum_{j}\pi_{ij} & = \mu_{i},\quad 
  \sum_{i}\pi_{ij} = \nu_{j}
\end{align*}
where we used the notation $\oplus$ to denote the outer sum, i.e. $\alpha\oplus\beta\in\RR^{N\times M}$ with $(\alpha\oplus\beta)_{ij} = \alpha_{i}+\beta_{j}$.

\begin{remark}
  Note the similarity to entropic regularization: There one can show that a plan $\pi$ is optimal if it is of the form $\pi = \exp(\tfrac{\alpha\oplus\beta-c}\gamma)$ and has correct row and column sums.
\end{remark}

An alternative formulation of the optimality system is: $\pi = (\rho+\alpha\oplus\beta-c)/\gamma$ is optimal if
\begin{align*}
  \rho & = (\alpha\oplus\beta-c)_{-}\\
  \sum_{j}(\rho_{ij} + \alpha_{i}+\beta_{j}-c_{ij}) & = \gamma\mu_{i}\\
  \sum_{i}(\rho_{ij} + \alpha_{i}+\beta_{j}-c_{ij}) & = \gamma\nu_{j}\\
\end{align*}
This leads us to a very simple algorithm: Initialize $\alpha$ and $\beta$ and cyclically solve the first equation above for $\rho$, the second for $\alpha$, and the third for $\beta$. This algorithm is described as Algorithm~\ref{alg:qrot_cp}.
Note that we can interpret Algorithm~\ref{alg:qrot_cp} as a cyclic projection method: The quadratically regularized optimal transport problem~\eqref{eq:qrot-discrete} is equivalent to minimizing $\norm{-\tfrac{c}{\gamma}-\pi}_{2}^{2}$ over the constraints $\pi\geq0$, $\pi\ones=\mu$, and $\pi^{T}\ones =\nu$, i.e. the solution is the projection of $-c/\gamma$ onto the set defined by these three constraints. Algorithm~\ref{alg:qrot_cp} does implicitly project $\pi$ cyclically onto these three constraints (without actually forming $\pi$ during the iteration). While iterative cyclic projections are guaranteed to find a feasible point, it is not guaranteed that the iteration converges to the projection in general~\cite{bauschke1997method}. However, in this case the fixed points $\alpha^{*}$, $\beta^{*}$ of the algorithm are indeed solutions of~\eqref{eq:qrot-discrete}, since the resulting $\pi = (\alpha^{*}\oplus\beta^{*}-c)_{+}/\gamma$ has the correct form and marginals.

\begin{algorithm}
\caption{Cyclic projection for quadratically regularized optimal transport}\label{alg:qrot_cp}
\begin{algorithmic}
\State Initialize: $\alpha^{0}=0\in\RR^{N}$, $\beta^{0}=0\in\RR^{M}$, set $n=0$
\Repeat
\State $\rho^{n+1}_{ij} = (\alpha_{i}^{n} + \beta_{j}^{n} - c_{ij})_{-}$
\State $\alpha^{n+1}_{i} = \tfrac{\gamma}M\Big(\mu_{i} - \tfrac1\gamma\sum_{j}\rho_{ij}^{n+1}+\beta_{j}^{n}-c_{ij}\Big)$
\State $\beta^{n+1}_{j} = \tfrac{\gamma}N\Big(\nu_{j} - \tfrac1\gamma\sum_{i}\rho_{ij}^{n+1}+\alpha_{i}^{n+1}-c_{ij}\Big)$
\State $n \gets n+1$
\Until{some stopping criterion}
\State Output $\pi = (\alpha^{n}\oplus\beta^{n}-c)/\gamma$
\end{algorithmic}
\end{algorithm}

Another natural choice for an algorithm is the gradient method on the dual problem of~\eqref{eq:qrot-discrete}, namely on
\[
\min_{\alpha,\beta}\left\{F(\alpha,\beta) := \tfrac12\norm{(\alpha\oplus\beta-c)_{+}}_{2}^{2} -\gamma\scp{\alpha}{\mu} - \gamma\scp{\beta}{\nu}\right\}.
\]
The gradients with respect to $\alpha$ and $\beta$ are
\[
\nabla_{\alpha}F(\alpha,\beta) = (\sum_{j}(\alpha_{i}+\beta_{j}-c_{ij})_{+}-\gamma\mu_{i}),\quad 
\nabla_{\beta}F(\alpha,\beta) = (\sum_{i}(\alpha_{i}+\beta_{j}-c_{ij})_{+}-\gamma\nu_{j}),
\]
respectively. With the help of the plans $\pi = (\alpha\oplus\beta-c)_+/\gamma$ one can express the gradients as $\nabla_{\alpha}F = \gamma(\pi\ones-\mu)$ and $\nabla_{\beta}F = \gamma(\pi^{T}\ones-\nu)$, respectively. A natural stepsize that leads to good performance is $\tau = 1/(M+N)$.
This amounts to Algorithm~\ref{alg:qrot_gd}.
\begin{algorithm}
\caption{Dual gradient descent for quadratically regularized optimal transport}\label{alg:qrot_gd}
\begin{algorithmic}
\State Initialize: $\alpha^{0}\in\RR^{N}$, $\beta^{0}\in\RR^{M}$, stepsize $\tau=1/(M+N)$, set $n=0$
\Repeat
\State $\pi^{n} = (\alpha^{n} \oplus \beta^{n} - c)_{+}/\gamma$
\State $\alpha^{n+1} = \alpha^{n} -\tau\gamma(\pi^{n}\ones-\mu)$
\State $\beta^{n+1} = \beta^{n} -\tau\gamma\big((\pi^{n})^{T}\ones-\nu\big)$
\State $n \gets n+1$
\Until{some stopping criterion}
\end{algorithmic}
\end{algorithm}

Algorithm~\ref{alg:qrot_fp} below is another algorithm which works with extremely low cost per iteration.
It can be derived as follows: The gradients are differentiable almost everywhere and the Hessian of $F$ is
\[
G(\alpha,\beta) = 
  \begin{pmatrix}
    \diag(\sigma\ones) & \sigma\\
    \sigma^{T} & \diag(\sigma^{T}\ones)
  \end{pmatrix},\quad\text{with}\quad
  \sigma_{ij} =
  \begin{cases}
    1 & \alpha_{i} + \beta_{j}-c_{ij}\geq 0\\
    0 & \text{otherwise.}
  \end{cases}
\]
The (semismooth) Newton method from~\cite{Lorenz:2019} performs updates of the form
\[
\begin{pmatrix}
  \alpha^{n+1}\\\beta^{n+1}
\end{pmatrix}
= 
\begin{pmatrix}
  \alpha^{n}\\\beta^{n}
\end{pmatrix}
- G(\alpha^{n},\beta^{n})^{-1}
\begin{pmatrix}
  \nabla_{\alpha}F(\alpha^{n},\beta^{n})\\\nabla_{\beta}F(\alpha^{n},\beta^{n})
\end{pmatrix}.
\]
To reduce the computation, we can omit the inversion of $G$, by replacing it with the simpler matrix
\[
M =
\begin{pmatrix}
  M(I + \tfrac1N\ones)& 0\\
  0 & N(I + \tfrac1M\ones)
\end{pmatrix},\quad\text{with}\quad M^{-1} = 
\begin{pmatrix}
  \tfrac1M(I - \tfrac1{2N}\ones) &  0\\
  0 & \tfrac1N(I - \tfrac1{2M}\ones)
\end{pmatrix}.
\]
where $I$ denotes the identity matrix and $\ones$ denotes the matrix of all ones (of appropriate sizes).

\begin{lemma}
  Fixed points of Algorithm~\ref{alg:qrot_fp} are optimal solutions of the quadratically regularized optimal transport problem~\eqref{eq:qrot-discrete}.
\end{lemma}
\begin{proof}
  The vectors $\alpha$, $\beta$ are fixed points if and only if $f$ and
  $g$ are zero. But this means that $\pi\ones = \mu$ and
  $\pi^{T}\ones = \nu$ which is, by definition of $\pi$ in the
  algorithm, the optimality condition.  This shows that fixed points
  are optimal.
\end{proof}
Note that Algorithm~\ref{alg:qrot_fp} is very similar to the dual gradient descent in Algorithm~\ref{alg:qrot_gd} (it mainly differs in the stepsizes and the subtraction of the mean values).

\begin{algorithm}
\caption{Simple fixed point iteration for quadratically regularized optimal transport}\label{alg:qrot_fp}
\begin{algorithmic}
\State Initialize: $\alpha^{0}\in\RR^{N}$, $\beta^{0}\in\RR^{M}$, set $n=0$
\Repeat
\State $\pi^{n}_{ij} = (\alpha_{i}^{n} + \beta_{j}^{n} - c_{ij})_{+}/\gamma$
\State $f_{i}^{n} = -\gamma(\sum_{j}\pi_{ij}^{n} - \mu_{i})$
\State $\alpha^{n+1}_{i} = \alpha_{i}^{n} + \tfrac1M\left(f_{i}^{n} - \tfrac{\sum_{i}f_{i}^{n}}{2N}\right)$
\State $g_{j}^{n} = -\gamma(\sum_{i}\pi_{ij}^{n} - \nu_{j})$
\State $\beta_{j}^{n+1} = \beta_{j}^{n} + \tfrac1N\left( g_{j}^{n} - \tfrac{\sum_{j}g_{j}^{n}}{2M}\right)$
\State $n \gets n+1$
\Until{some stopping criterion}
\end{algorithmic}
\end{algorithm}

As a final algorithm we tested Nesterov's accelerated gradient descent of the dual as stated in Algorithm~\ref{alg:qrot_ng}. We used the same stepsize as for Algorithm~\ref{alg:qrot_gd}.

\begin{algorithm}
\caption{Dual Nesterov gradient descent for quadratically regularized optimal transport}\label{alg:qrot_ng}
\begin{algorithmic}
\State Initialize: $\alpha^{0}=\alpha^{-1}\in\RR^{N}$, $\beta^{0}=\beta^{-1}\in\RR^{M}$, stepsize $\tau=1/(M+N)$, set $n=0$
\Repeat
\State $\bar\alpha^{n} = \alpha^{n} + \sigma_{n}(\alpha^{n}-\alpha^{n-1})$, $\bar\beta^{n} = \beta^{n} + \sigma_{n}(\beta^{n}-\beta^{n-1})$ with $\sigma_{n} = n/(n+3)$
\State $\pi^{n} = (\bar\alpha^{n} \oplus \bar\beta^{n} - c)_{+}/\gamma$
\State $\alpha^{n+1} = \bar\alpha^{n} -\tau\gamma(\pi^{n}\ones-\mu)$
\State $\beta^{n+1} = \bar\beta^{n} -\tau\gamma\big((\pi^{n})^{T}\ones-\nu\big)$
\State $n \gets n+1$
\Until{some stopping criterion}
\end{algorithmic}
\end{algorithm}

Although the pseudo-code for all algorithms explicitly forms the outer sums $\alpha\oplus\beta$ at some points, this is not needed in implementations. In all cases we only need row- and colum-sums of these larger quantities of size $N\times M$ and these can be computed in parallel.

Figure~\ref{fig:example_1d_quad} shows results for simple one-dimensionals marginals and quadratic cost function $c(x,y) = |x-y|^{2}$ and in Figure~\ref{fig:example_1d_abs} we used the absolute value $c(x,y) = |x-y|$. In all examples the cyclic projection (Algorithm~\ref{alg:qrot_cp}) and the fixed-point iteration (Algorithm~\ref{alg:qrot_fp}) perform good (Algorithm~\ref{alg:qrot_fp} always slightly ahead) while dual gradient descent (Algorithm~\ref{alg:qrot_gd}) is always significantly slower. Nesterov's gradient descent (Algorithm~\ref{alg:qrot_ng}) oscillates heavily, takes longer to reduce the error in the beginning but keeps reducing the error faster than the other methods.

\begin{figure}
  \centering
    \begin{subfigure}[b]{0.45\textwidth}
      \includegraphics[width=\textwidth]{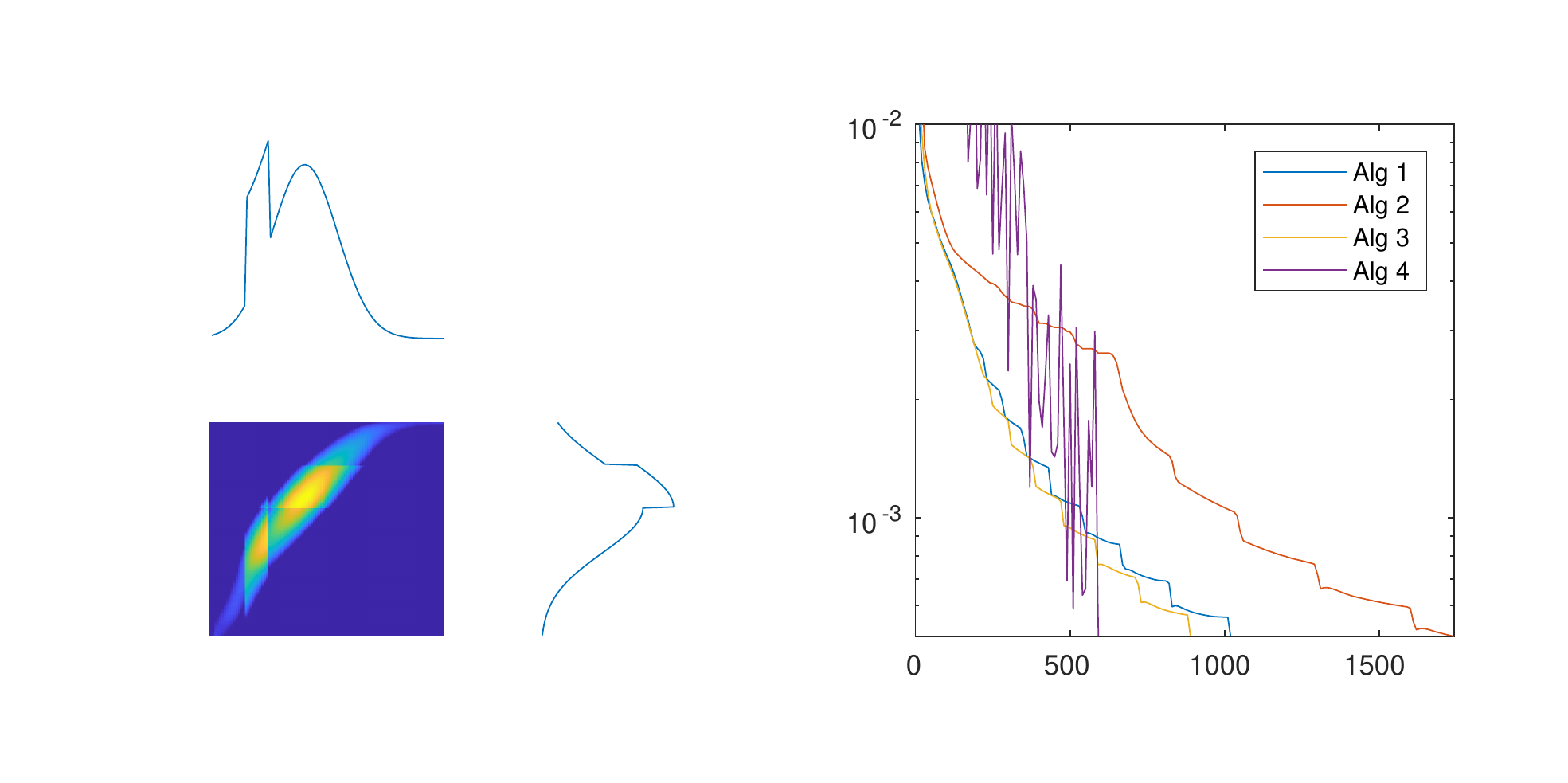}\caption{$\gamma=50$}
    \end{subfigure}
    \begin{subfigure}[b]{0.45\textwidth}
      \includegraphics[width=\textwidth]{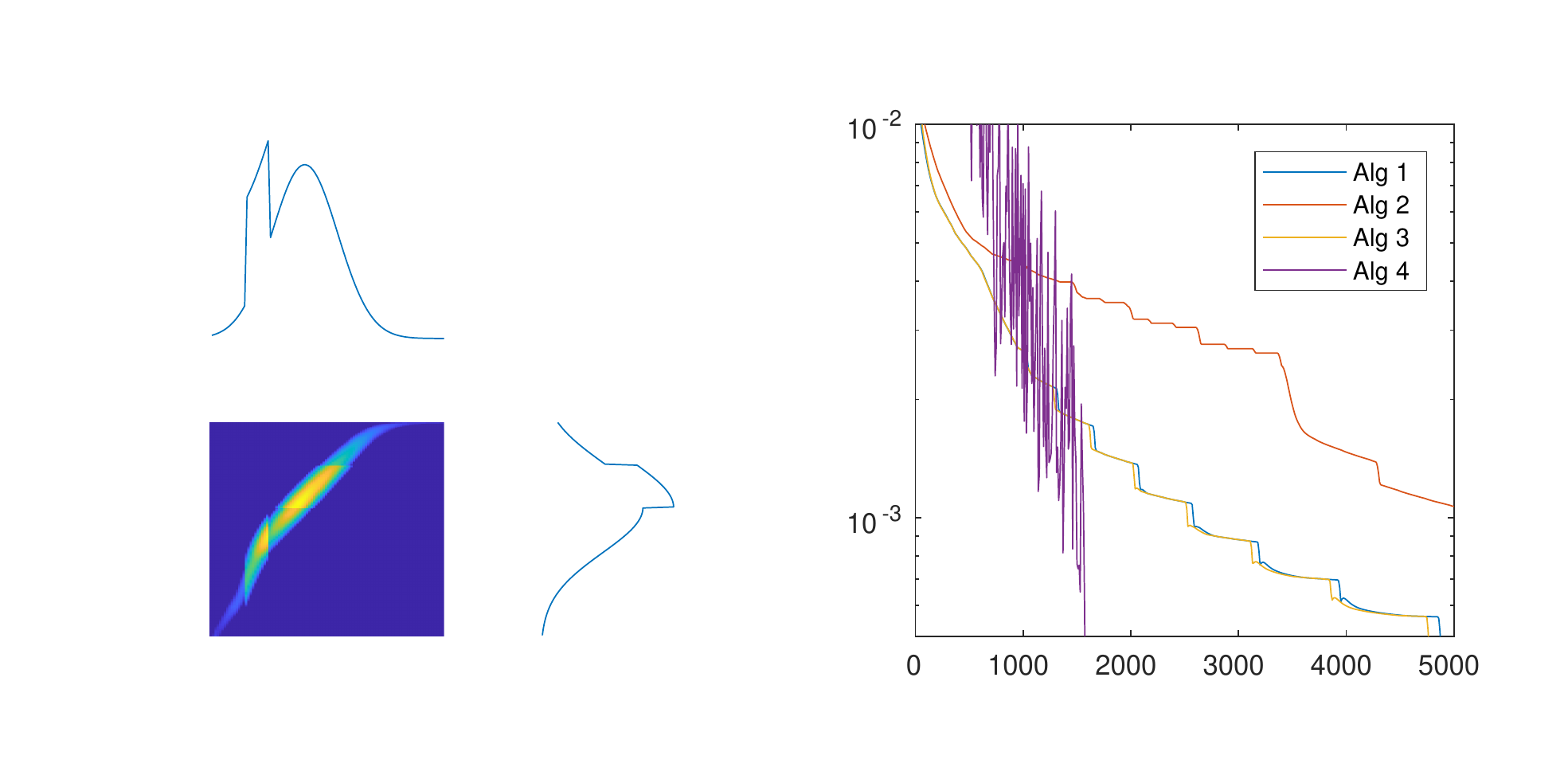}\caption{$\gamma=10$}
    \end{subfigure}
    \begin{subfigure}[b]{0.45\textwidth}
      \includegraphics[width=\textwidth]{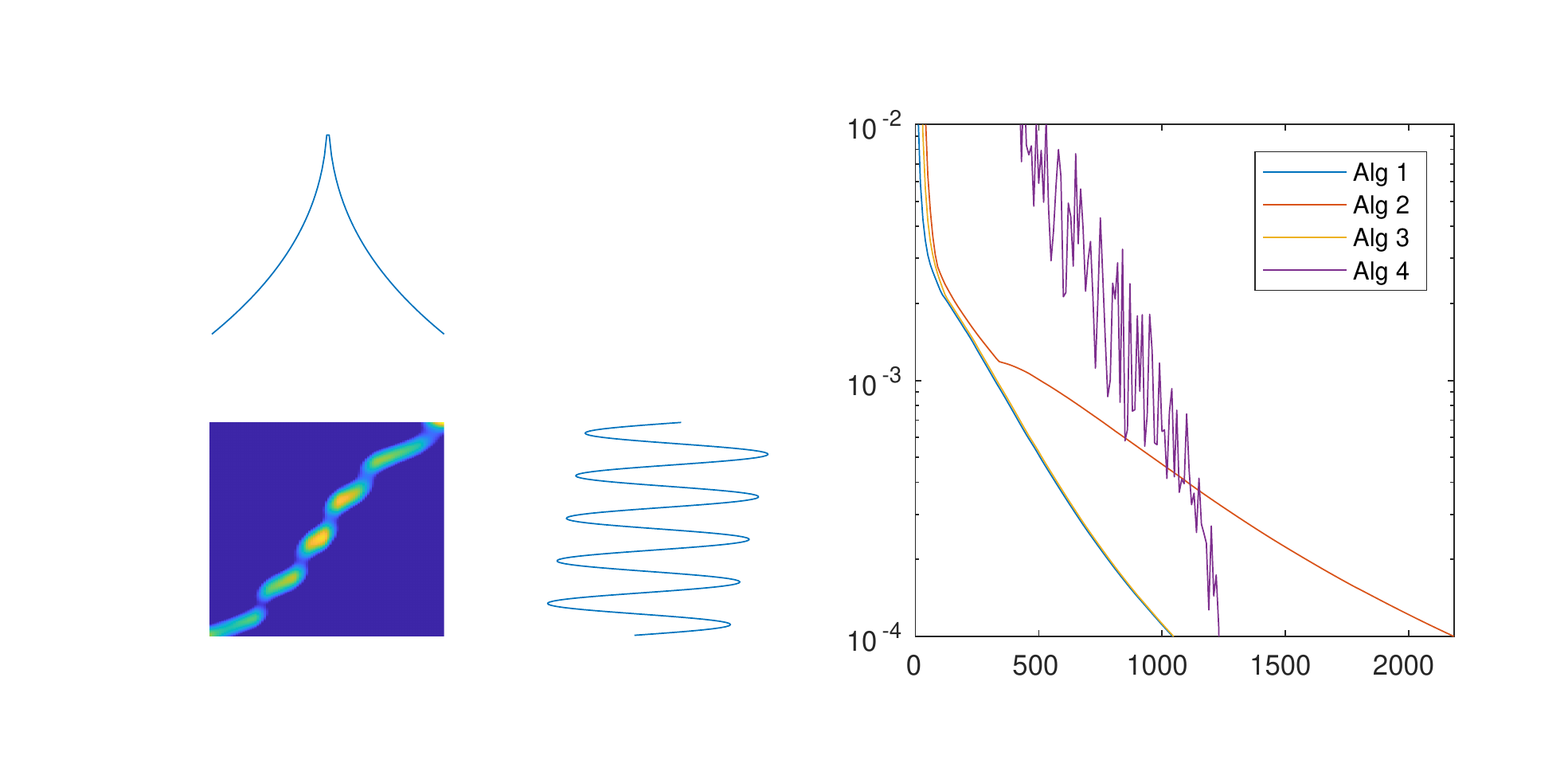}\caption{$\gamma=10$}
    \end{subfigure}
    \begin{subfigure}[b]{0.45\textwidth}
      \includegraphics[width=\textwidth]{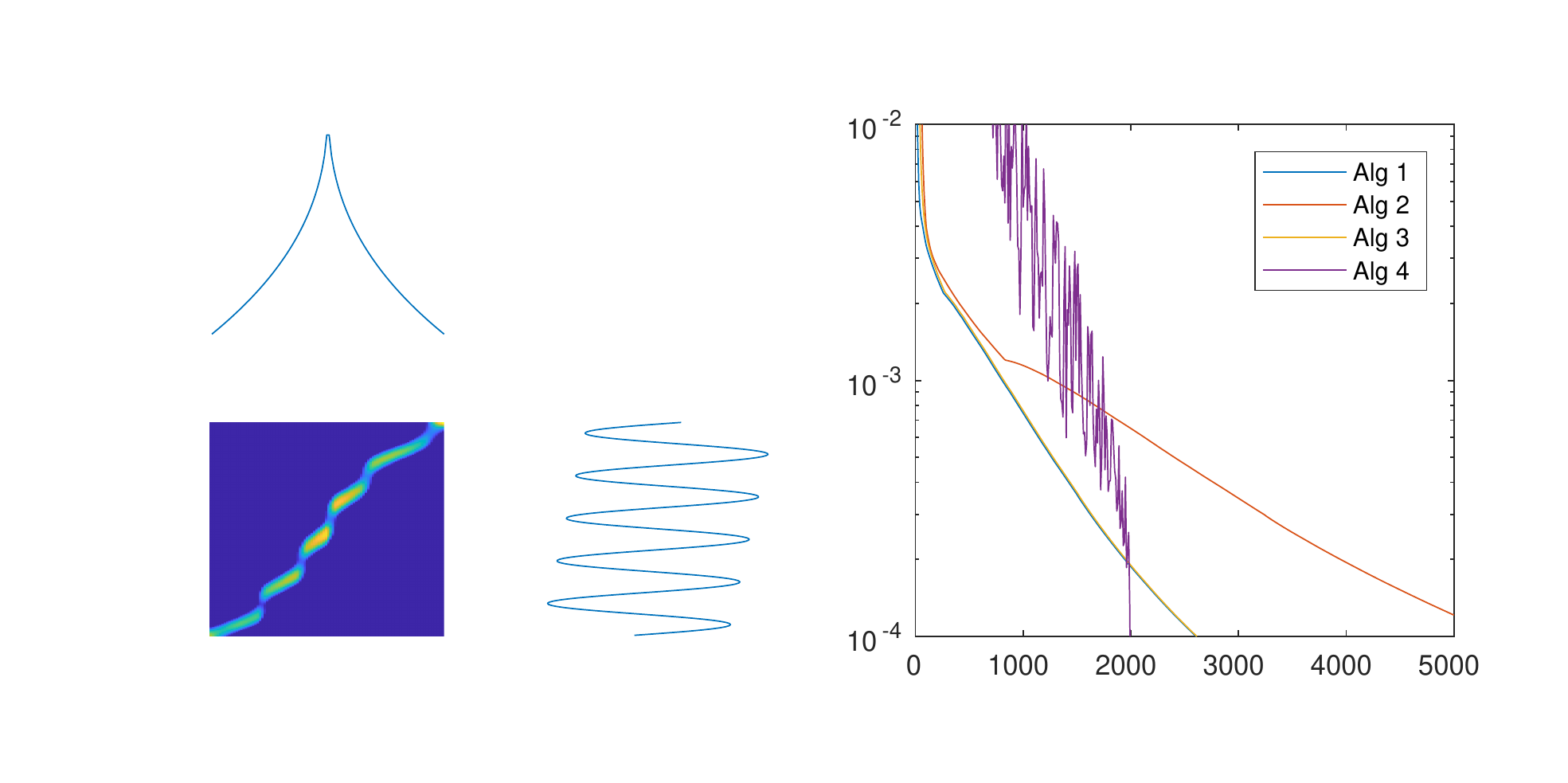}\caption{$\gamma=4$}
    \end{subfigure}
  \caption{Results for one-dimensional examples with $c(x,y) = |x-y|^{2}$. Each subfigure on the left: Marginals and optimal plan. Each subfigure on the right: maximal violation of constraints over iteration count.}
  \label{fig:example_1d_quad}
\end{figure}%

\begin{figure}
  \centering
    \begin{subfigure}[b]{0.45\textwidth}
      \includegraphics[width=\textwidth]{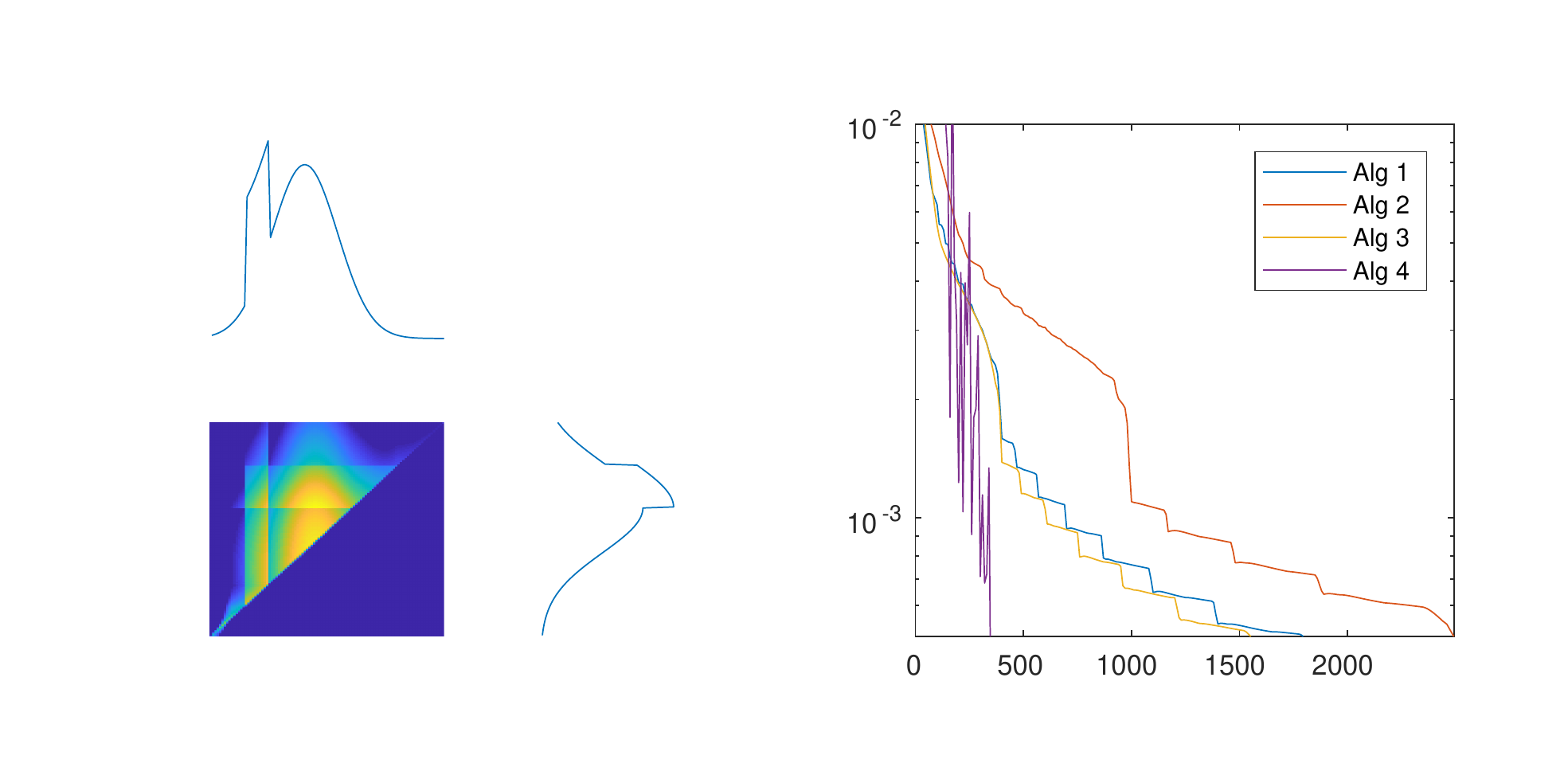}\caption{$\gamma=100$}
    \end{subfigure}
    \begin{subfigure}[b]{0.45\textwidth}
      \includegraphics[width=\textwidth]{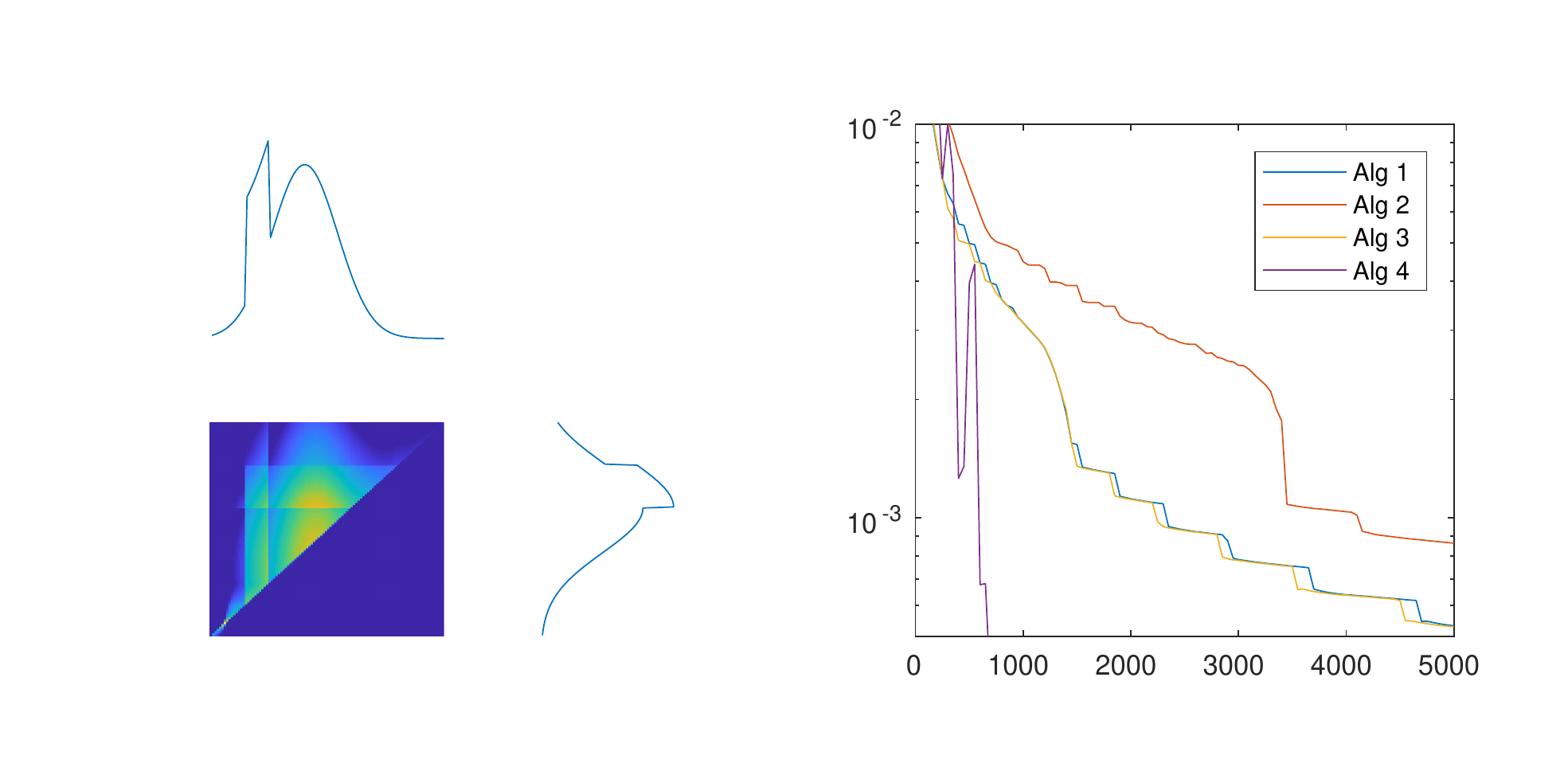}\caption{$\gamma=50$}
    \end{subfigure}
    \begin{subfigure}[b]{0.45\textwidth}
      \includegraphics[width=\textwidth]{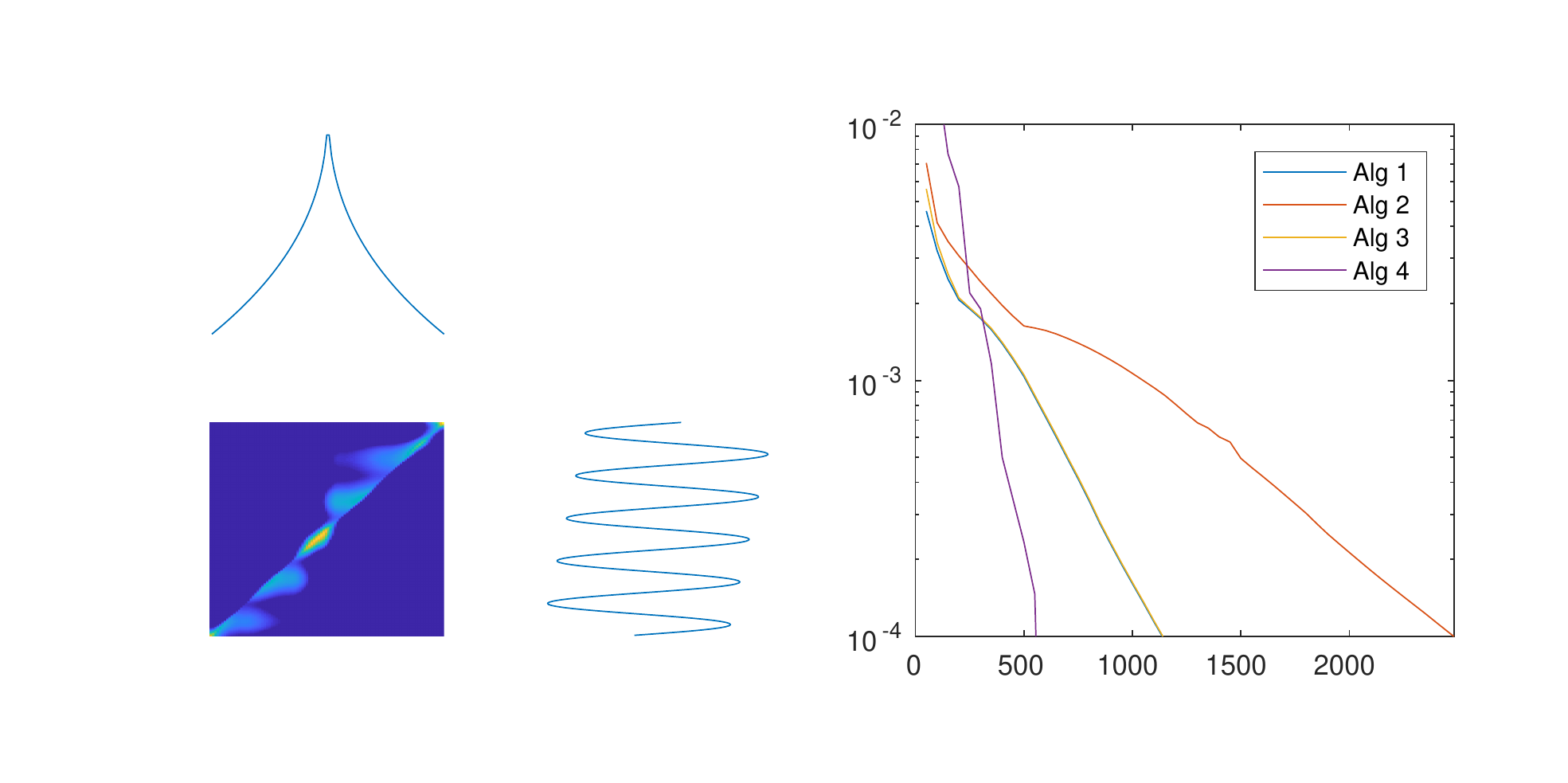}\caption{$\gamma=50$}
    \end{subfigure}
    \begin{subfigure}[b]{0.45\textwidth}
      \includegraphics[width=\textwidth]{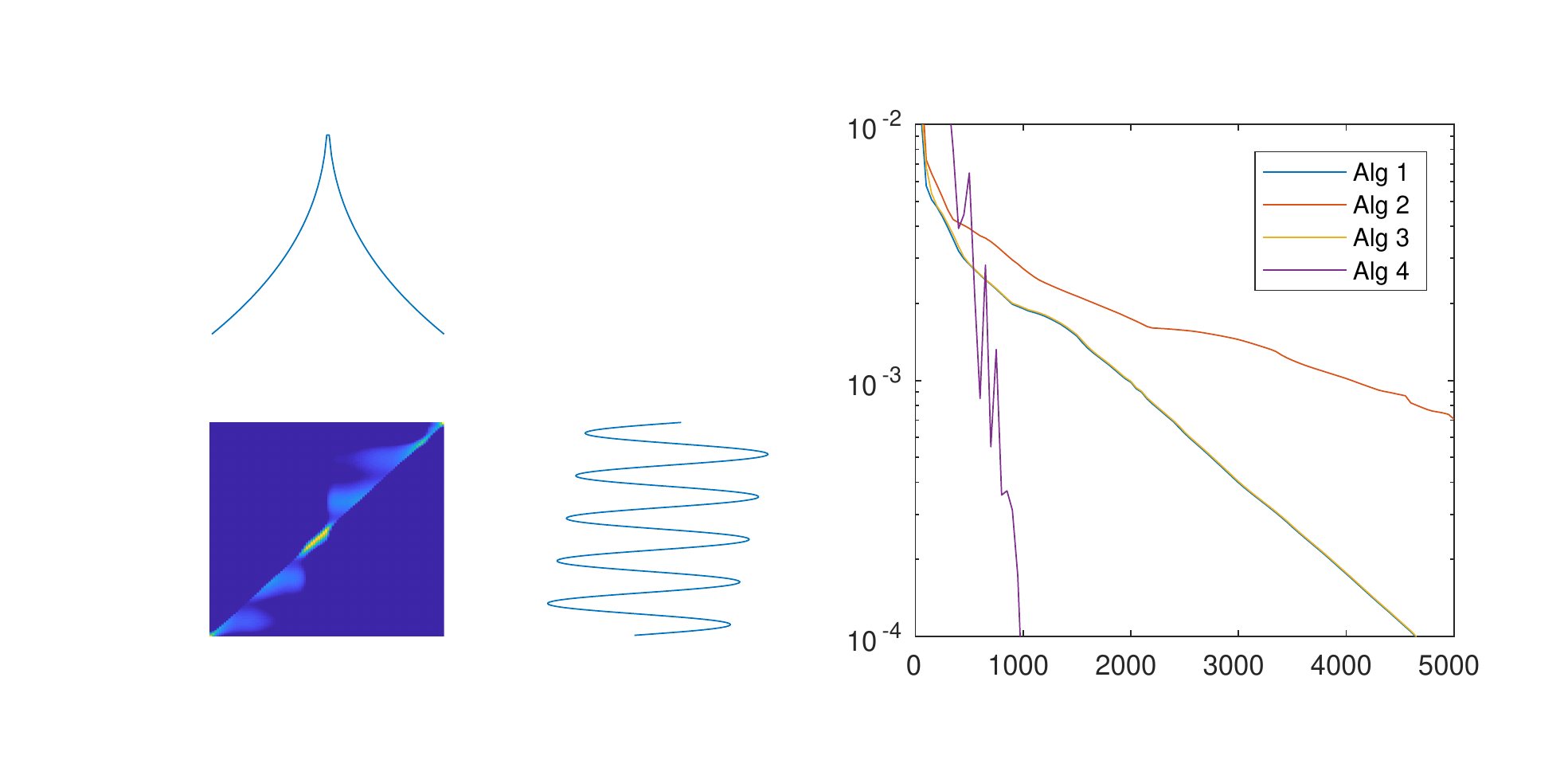}\caption{$\gamma=15$}
    \end{subfigure}
  \caption{Results for one-dimensional examples with $c(x,y) = |x-y|$. Each subfigure on the left: Marginals and optimal plan. Each subfigure on the right: maximal violation of constraints over iteration count.}
  \label{fig:example_1d_abs}
\end{figure}%

\small

%\setcitestyle{numberical}
\bibliographystyle{plain}
\bibliography{literature}

\end{document}